\documentclass[a4paper, 12pt]{amsart}
\usepackage{tikz}
\usepackage{enumerate,hyperref}

\newtheorem{thm}{Theorem}[section]
\newtheorem{prp}[thm]{Proposition}
\newtheorem{lem}[thm]{Lemma}
\newtheorem{cor}[thm]{Corollary}
\theoremstyle{definition}
\newtheorem{dfn}[thm]{Definition}
\theoremstyle{remark}
\newtheorem{rmk}[thm]{Remark}

\newcommand{\go}{G^{(0)}}
\newcommand{\ho}{H^{(0)}}
\newcommand{\lsp}{\operatorname{span}}
\newcommand{\clsp}{\overline{\lsp}}

\newcommand{\Kk}{\mathcal{K}}
\newcommand{\Cc}{\mathcal{C}}
\newcommand{\Rr}{\mathcal{R}}
\newcommand{\Uu}{\mathcal{U}}

\newcommand{\NN}{\mathbb{N}}
\newcommand{\ZZ}{\mathbb{Z}}

\newcounter{saver}

\author[T.M. Carlsen]{Toke Meier Carlsen}
\address[T.M. Carlsen]{Department of Science and Technology\\University of the Faroe Islands
N\'oat\'un 3\\ FO-100 T\'orshavn\\the Faroe Islands}
\email{toke.carlsen@gmail.com}

\author[E. Ruiz]{Efren Ruiz}
\address[E. Ruiz]{Department of Mathematics\\University of Hawaii,
Hilo\\200 W. Kawili St.\\
Hilo, Hawaii\\
96720-4091 USA}
\email{ruize@hawaii.edu}

\author[A. Sims]{Aidan Sims}
\address[A. Sims]{School of Mathematics and Applied Statistics\\
University of Wollongong\\
NSW 2522\\
Australia}
\email{asims@uow.edu.au}

\title[Diagonal-preserving stable isomorphism]{Equivalence and stable isomorphism of groupoids,
and diagonal-preserving stable isomorphisms of graph $C^*$-algebras and Leavitt path algebras}
\subjclass[2010]{46L05 (primary); 16S99 (secondary)}
\keywords{Stabilisation; stable isomorphism; groupoid equivalence; Kakutani equivalence; graph algebra; Leavitt path algebra}

\date{\today}

\thanks{This work was initiated, and much of it completed, while all three authors were
attending the research program \emph{Classification of operator algebras: complexity,
rigidity, and dynamics} at the Mittag-Leffler Institute, January--April 2016.
This research was supported by Australian Research Council grant DP150101598 and by a grant from the Simons Foundation
(\#279369 to Efren Ruiz). We thank Dana Williams for helpful email correspondence.}

\begin{document}

\begin{abstract}
We prove that ample groupoids with $\sigma$-compact unit spaces are equivalent if and
only if they are stably isomorphic in an appropriate sense, and relate this to Matui's
notion of Kakutani equivalence. We use this result to show that diagonal-preserving
stable isomorphisms of graph $C^*$-algebras or Leavitt path algebras give rise to
isomorphisms of the groupoids of the associated stabilised graphs. We deduce that the
Leavitt path algebras $L_\ZZ(E_2)$ and $L_\ZZ(E_{2-})$ are not stably $^*$-isomorphic.
\end{abstract}

\maketitle

\section{Introduction}

A beautiful recent theorem of Matsumoto and Matui \cite{MatsumotoMatui} relates
diagonal-preserving isomorphism of Cuntz--Krieger algebras to the Bowen--Franks
invariants of the corresponding shifts of finite type, and to isomorphism of the
associated graph groupoids. As a result, diagonal-preserving isomorphism has become an
important notion in structure theory for graph $C^*$-algebras and Leavitt path algebras
\cite{JohansenSorensen, Carlsen}. A key ingredient in Matsumoto and Matui's approach is
the Weyl-groupoid construction, which reconstructs a groupoid from an associated algebra
and diagonal subalgebra. This construction goes back to the work of Feldman and Moore
\cite{FeldmanMoore} on von Neumann factors and was continued by Kumjian \cite{Kumjian}
and Renault \cite{Renault} for $C^*$-algebras. More recently, it has been refined by
Brownlowe--Carlsen--Whittaker \cite{BCW} for graph $C^*$-algebras, by Brown--Clark--an
Huef \cite{BCaH} for Leavitt path algebras, and by Ara--Bosa--Hazrat--Sims \cite{ABHS}
for Steinberg algebras.

The Weyl-groupoid approach is well-suited to questions about isomorphisms of graph
$C^*$-algebras or of Leavitt path algebras. But to use it to study stable isomorphism,
one first needs a groupoid-theoretic analogue of the Brown--Green--Rieffel
stable-isomorphism theorem for $C^*$-algebras. Here we supply such a theorem
(Theorem~\ref{thm:groupoidBGR}), and explore its consequences for graph $C^*$-algebras
and Leavitt path algebras (Section~\ref{sec:consequences}).

\smallskip

We begin in Section~\ref{sec:BGR} by proving our Brown--Green--Rieffel theorem for ample
groupoids with $\sigma$-compact unit spaces. We do not assume that our groupoids are
Hausdorff or second countable. Our proof parallels Brown's proof that a full corner of a
$\sigma$-unital $C^*$-algebra is stably isomorphic to the enveloping algebra. In
Section~\ref{sec:Kakutani} we digress to relate our results to Matui's definition
\cite{Matui} of Kakutani equivalence for ample groupoids with compact unit space. We
extend this notion to ample groupoids with noncompact unit space and prove that it
coincides with groupoid equivalence. We start Section~\ref{sec:consequences} by checking
that Tomforde's construction from a directed graph $E$ of a graph $SE$ satisfying
$C^*(SE) \cong C^*(E) \otimes \Kk$ is compatible with stabilising the groupoid. We then
explore the consequences of our Brown--Green--Rieffel theorem for groupoid $C^*$-algebras and
Steinberg algebras, and particularly for graph $C^*$-algebras and Leavitt path algebras:
Theorem~\ref{thm:main} says, amongst other things, that there is a diagonal-preserving
isomorphism $C^*(E) \otimes \Kk \cong C^*(F) \otimes \Kk$ if and only if there is a
diagonal-preserving isomorphism $C^*(SE) \cong C^*(SF)$, and likewise at the level of
Leavitt path algebras. We deduce using results of Carlsen \cite{Carlsen} that
$L_\ZZ(E_2)$ and $L_\ZZ(E_{2-})$ are not stably $^*$-isomorphic.

\section{Groupoid equivalence and stable isomorphism}\label{sec:BGR}

In this section we show that for ample groupoids, the Brown--Green--Rieffel
stable-isomorphism theorem \cite{BrownGreenRieffel:PJM77} works at the level of
groupoids.

An ample groupoid is a groupoid $G$ equipped with a topology with a basis of compact open
sets such that inversion and composition in $G$ are continuous, the unit space $\go$ is
Hausdorff, and the range and source maps $r,s : G \to \go$ are local homeomorphisms. The
unit space of an ample groupoid is automatically locally compact and totally
disconnected.

For groupoids $G$ and $H$, a \emph{$G$--$H$ equivalence} is a space $Z$ with commuting
free and proper actions of $G$ on the left and $H$ on the right such that $r : Z \to \go$
induces a homeomorphism $Z/H \cong \go$ and $s : Z \to \ho$ induces a homeomorphism of
$G\backslash Z \cong \ho$; if such a $Z$ exists, we say that $G$ and $H$ are \emph{groupoid
equivalent}. See \cite{MRW, Renault82} for more detail.

We will write $\Rr$ for the full countable equivalence relation $\Rr = \NN \times \NN$,
regarded as a discrete principal groupoid with unit space $\NN$. A space is
\emph{$\sigma$-compact} if it has a countable cover by compact sets; if it is locally
compact, totally disconnected and Hausdorff, it then has a countable cover by mutually
disjoint compact open sets. Given an ample groupoid $G$, the product $G \times \Rr$ is an
ample groupoid under the product topology and coordinatewise operations. We identify the
unit space of $G \times \Rr$ with $G^{(0)}\times\NN$.

\begin{thm}\label{thm:groupoidBGR}
Let $G$ and $H$ be ample groupoids. Suppose that $\go$ and $\ho$ are $\sigma$-compact.
Then $G$ and $H$ are groupoid equivalent if and only if $G \times \Rr \cong H \times
\Rr$.
\end{thm}

The strategy is to prove that for any clopen $K \subseteq \go$ that meets every
$G$-orbit, $G \times \Rr \cong G|_K \times \Rr$, paralleling Brown's result about full
corners of $\sigma$-unital $C^*$-algebras. Our proof follows Brown's very closely:
Lemmas~\ref{lem:lots-o-bisections}, \ref{lem:isometry}~and~\ref{lem:unitary} and their
proofs are direct analogues of \cite[Lemmas~2.3, 2.4~and~2.5]{Brown:PJM77}.

We say that $U \subseteq G$ is an \emph{open bisection} if $U$ is open and $r, s$
restrict to homeomorphisms of $U$ onto $r(U), s(U)$ respectively. For $x \in \go$, we
denote $r^{-1}(x)$ by $G^x$ and $s^{-1} (x)$ by $G_x$, and for $K \subseteq \go$, we
write $G K := s^{-1}(K)$, $K G := r^{-1}(K)$, and $G|_K := KG \cap GK$. A set $K
\subseteq \go$ is \emph{$G$-full} if $r(GK) = \go$.

\begin{lem}\label{lem:lots-o-bisections}
Let $G$ be an ample groupoid such that $\go$ is $\sigma$-compact. Suppose that $K
\subseteq \go$ is clopen and $G$-full. Then there is a sequence of compact open
bisections $V_i \subseteq G K$ with mutually disjoint ranges such that $\bigsqcup_i
r(V_i) = \go$.
\end{lem}
\begin{proof}
Choose a countable cover $\Uu$ of $\go$ by compact open sets. Fix $U \in \Uu$. For $u \in
U$, since $K$ is $G$-full, there exists $\gamma_u \in G^u \cap GK$. Since $K$ is open and
$G$ is ample, for each $u \in U$, there is a compact open bisection $V_u$ such that
$\gamma_u \in V_u \subseteq G K$. Each $r(V_u)$ is clopen in $\go$ because $\go$ is
Hausdorff. Since $U$ is compact, we can find $V_{u_1}, \dots V_{u_{j(U)}}$ with $U
\subseteq \bigcup^{j(U)}_{i=1} r(V_{u_i})$. By choosing a finite collection like this for
each $U \in \Uu$ and enumerating the union of these collections, we obtain a list
$(V^0_i)^\infty_{i=1}$ of compact open bisections with $\bigcup r(V^0_i) = \go$ and
$\bigcup s(V^0_i) \subseteq K$. For each $i$, the set $X_i := r(V^0_i) \setminus
\bigcup_{j < i} r(V^0_j)$ is compact open in $\go$. Since $r^{-1} : r(V^0_i) \to V^0_i$
is a homeomorphism we deduce that $V_i := V^0_i \cap X_i G$ is a compact open subset of
$V^0_i$. These $V_i$ suffice.
\end{proof}

\begin{lem}\label{lem:isometry}
Let $G$ be an ample groupoid such that $\go$ is $\sigma$-compact. Suppose that $K
\subseteq \go$ is clopen and $G$-full. Then there is an open bisection $W \subseteq G
\times \Rr$ such that $r(W) = \go \times \{1\}$ and $s(W) \subseteq K \times \NN$ is
clopen in $\go \times \NN$.
\end{lem}
\begin{proof}
Fix compact open bisections $(V_i)^\infty_{i=1}$ as in Lemma~\ref{lem:lots-o-bisections}.
Put $W := \bigcup_i V_i \times \{(1,i)\}$, which is open because the $V_i$ are. The
$r(V_i \times \{(1,i)\})$ are mutually disjoint because the $r(V_i)$ are; the $s(V_i
\times \{(1,i)\})$ are clearly mutually disjoint. The maps $s,r$ are homeomorphisms on
$W$ because they restrict to homeomorphisms on the relatively clopen subsets $V_i \times
\{(1,i)\}$. Clearly $s(W) = \bigcup_i s(V_i) \times \{i\}$ is open. It is also closed
because the $s(V_i)$ are closed in $\go$, so $(\go \times \NN) \setminus s(W) = \bigcup_i \big((\go
\setminus s(V_i)) \times \{i\})$ is open.
\end{proof}

\begin{lem}\label{lem:unitary}
Under the hypotheses of Lemma~\ref{lem:isometry}, there is an open bisection $Y \subseteq
G \times \Rr$ such that $r(Y) = \go \times \NN$ and $s(Y) = K \times \NN$.
\end{lem}
\begin{proof}
Write $\NN = \bigsqcup_{i=1}^\infty N_i$ as a union of mutually disjoint infinite
subsets. We claim that there exists a sequence $Y_j$ of open bisections with mutually
disjoint clopen ranges and mutually disjoint clopen sources such that for each $n \ge 0$,
we have
\begin{align*}
\textstyle\bigcup^{2n-1}_{j=1} r(Y_j) &=\textstyle \bigcup^n_{i=1} \go \times N_i,
    &&&\textstyle \bigcup^{2n}_{j=1} r(Y_j) &\textstyle\subseteq \bigcup^{n+1}_{i=1} \go \times N_i,\\[0.25em]
\textstyle\bigcup^{2n-1}_{j=1} s(Y_j) &\textstyle\subseteq \bigcup^n_{i=1} K \times N_i,&\text{and}&
    &\textstyle \bigcup^{2n}_{j=1} s(Y_j) &\textstyle= \bigcup^{n}_{i=1} K \times N_i.
\end{align*}

Suppose that $Y_1, \dots, Y_{2n}$ satisfy these equations (this is trivial when $n = 0$).
To construct $Y_{2n+1}$, apply Lemma~\ref{lem:isometry} to $G \times \left(
\mathcal{R}|_{N_{n+1}} \right) $ and $K \times N_{n+1} \subseteq \go \times N_{n+1}$ to
obtain an open bisection $W \subseteq G \times \left( \mathcal{R}|_{N_{n+1}} \right)
\times \mathcal{R}$ such that $r(W) = \go \times N_{n+1} \times \{1\}$ and $s(W)
\subseteq K \times N_{n+1} \times \NN$ is clopen. Fix a bijection $\theta : N_{n+1}
\times \NN \to N_{n+1}$, and define
\[
W' := \{(g, (p, \theta(q,m))) : (g, (p,q), (1,m)) \in W\}
    \subseteq G \times \left( \mathcal{R}|_{N_{n+1}} \right).
\]
This $W'$ is an open bisection with $r(W') = \go \times N_{n+1}$ and $s(W') \subseteq K
\times N_{n+1}$. Let
\[
Y_{2n+1} := \big((\go \times N_{n+1}) \setminus r(Y_{2n})\big) W'.
\]
Since $r(Y_{2n})$ is clopen as part of the induction hypothesis, so is $(\go \times
N_{n+1}) \setminus r(Y_{2n})$; so $Y_{2n+1}$ is open. Since $r$ and $s$ restrict to
homeomorphisms on $W'$, the set $s(Y_{2n+1})$ is clopen in $G^{(0)}\times\NN$. We
have $\bigcup^{2n+1}_{j=1} r(Y_j) = \bigcup^{n+1}_{i=1} \go \times N_i$ by definition of
$Y_{2n+1}$, and clearly $\bigcup^{2n+1}_{j=1} s(Y_j) \subseteq \bigcup^{n+1}_{i=1} K
\times N_i$.

To construct $Y_{2n+2}$, choose a bijection $\phi : N_{n+1} \to N_{n+2}$, and define
\begin{align*}
Y_{2n+2} &:=\textstyle
    \Big(\bigcup_{i \in N_{n+1}} \go \times \{(\phi(i),i)\}\Big)
        \big((K \times N_{n+1}) \setminus s(Y_{2n+1})\big)\\
    &= \big\{(u, \phi(n)) : (u,n) \in (K \times N_{n+1}) \setminus s(Y_{2n+1})\big\}
    \subseteq \go \times \Rr.
\end{align*}
This is open because $s(Y_{2n+1})$ is closed. It is a bisection because $\phi$ is a
bijection. Both $s(Y_{2n+2})$, $r(Y_{2n+2})$ are clopen in $G^{(0)}\times\NN$ because $s$
and $r$ are homeomorphisms on $\big(\bigcup_{i \in N_{n+1}} \go \times
\{(\phi(i),i)\}\big)$. We have $\bigcup^{2n+2}_{j=1} s(Y_j) = \bigcup^{n+1}_{i=1} K
\times N_i$ and $\bigcup^{2n+2}_{j=1} r(Y_j) \subseteq \go \times \bigcup^{n+2}_{i=1}
N_i$ by construction. This proves the claim.

Let $Y := \bigcup^\infty_{i=1} Y_i$, which is open because the $Y_i$ are open. Since the
$s(Y_i)$ are mutually disjoint, $s$ is injective on $Y$; and similarly for $r$. Since the
$Y_i$ are open and $s,r$ restrict to homeomorphisms on the $Y_i$, we see that $s,r$ are
homeomorphisms on $W$. We have $r(Y) = \bigcup_n \bigcup^{2n-1}_{j=1} r(Y_j) = \go \times
\NN$ and $s(Y) = \bigcup_n \bigcup^{2n}_{j=1} s(Y_j) = K \times \NN$.
\end{proof}

We now obtain a groupoid version of \cite[Corollary~2.6]{Brown:PJM77}.

\begin{prp}\label{prp:corner->stable iso}
Let $G$ be an ample groupoid such that $\go$ is $\sigma$-compact. Suppose that
$K\subseteq \go$ is clopen and $G$-full. Then $G \times \Rr \cong G|_K \times \Rr$.
\end{prp}
\begin{proof}
Apply Lemma~\ref{lem:unitary} to obtain an open bisection $Y \subseteq G \times \Rr$ such
that $r(Y) = \go \times \NN$ and $s(Y) = K \times \NN$. For $\gamma \in G \times \Rr$, we write
$Y^{-1}\gamma Y$ for the element $\alpha^{-1}\gamma\beta$ obtained from the unique
elements $\alpha,\beta \in Y$ with $r(\alpha) = r(\gamma)$ and $s(\beta) = s(\gamma)$.
Since $Y$ is a bisection, the map $\gamma \mapsto Y^{-1}\gamma Y$ is a groupoid
homomorphism with range in $G|_K \times \Rr$. It is continuous because multiplication in
$G \times \Rr$ is continuous. Since $\eta \mapsto Y \eta Y^{-1} : G|_K \times \Rr \to G \times \Rr$
is a continuous inverse, $\gamma \mapsto Y^{-1}\gamma Y$ is the desired isomorphism $G
\times \Rr \to G|_K \times \Rr$.
\end{proof}
\begin{proof}[Proof of Theorem~\ref{thm:groupoidBGR}]
Let $Z$ be a $G$--$H$-equivalence. Consider the linking groupoid $L = G \sqcup Z \sqcup
Z^{\operatorname{op}} \sqcup H$ \cite[Lemma~3]{SimsWilliams}. By
\cite[Lemma~4.2]{ClarkSims:JPAA15}, $\go, \ho \subseteq L^{(0)}$ both satisfy the
hypotheses of Proposition~\ref{prp:corner->stable iso}. So $G \times \Rr \cong L|_{\go}
\times \Rr \cong L \times \Rr \cong L|_{\ho} \times \Rr \cong H \times \Rr$.

Now suppose that $G \times \Rr \cong H \times \Rr$. The space $X := G \times \{(1,i) : i
\in \NN\}$ is a $G$--$(G \times \Rr)$-equivalence, and similarly $Z := H \times \{(i,1) :
i \in \NN\}$ is a $(H \times \Rr)$--$H$-equivalence. Since $G \times \Rr \cong H \times
\Rr$ and groupoid equivalence is an equivalence relation, we deduce that $G$ and $H$ are
groupoid equivalent.
\end{proof}

\section{Kakutani equivalence}\label{sec:Kakutani}

Matui \cite{Matui} defines Kakutani equivalence for ample groupoids $G$ and $H$ with
compact unit spaces: $G$ and $H$ are Kakutani equivalent if there are full clopen subsets
$X \subseteq \go$ and $Y \subseteq \ho$ such that $G|_X \cong H|_Y$. We extend this
notion to ample groupoids with non-compact unit spaces.

\begin{dfn}
Let $G$ and $H$ be ample groupoids. Then $G$ and $H$ are \emph{Kakutani equivalent} if
there are a $G$-full clopen $X \subseteq \go$ and an $H$-full clopen $Y \subseteq \ho$
such that $G|_X \cong H|_Y$.
\end{dfn}

\begin{thm}\label{thm:equiv equiv}
Let $G$ and $H$ be ample groupoids with $\sigma$-compact unit spaces. The following are
equivalent:
\begin{enumerate}
\item\label{it:Ke} $G$ and $H$ are Kakutani equivalent;
\item\label{it:wKe} there exist full open sets $X \subseteq \go$ and $Y \subseteq
    \ho$ such that $G|_X \cong H|_Y$;
\item\label{it:ge} $G$ and $H$ are groupoid equivalent;
\item $G \times \Rr \cong H \times \Rr$.
\end{enumerate}
\end{thm}
\begin{proof}
By Theorem~\ref{thm:groupoidBGR}, it suffices to show that (\ref{it:Ke})--(\ref{it:ge})
are equivalent. That (\ref{it:Ke})${\implies}$(\ref{it:wKe}) is obvious. Suppose
that~(\ref{it:wKe}) holds. Then $GX$ is a $G$--$G|_X$ equivalence under the actions
determined by multiplication in $G$ (see the argument of
\cite[Lemma~6.1]{ClarkSims:JPAA15}).  Similarly, $YH$ is a $H$--$H|_Y$ equivalence. Since
groupoid equivalence is an equivalence relation, $G$ and $H$ are groupoid equivalent,
giving (\ref{it:wKe})${\implies}$(\ref{it:ge}).

Now suppose that $Z$ is a $G$--$H$-equivalence. In this proof, for $K \subseteq \go$ we
write $[K]_G$ for the saturation $r(GK)$ of $K$ in $\go$; similarly, for $K' \subseteq
\ho$, we write $[K']_H := r(HK')$. Let $L = G \sqcup Z \sqcup Z^{\operatorname{op}}
\sqcup H$ be the linking groupoid \cite[Lemma~3]{SimsWilliams}. Fix countable covers $\go
= \bigsqcup^\infty_{i=1} U_i$ and $\ho = \bigsqcup^\infty_{i=1} W_i$ by mutually disjoint
compact open sets.

\textbf{Claim.} There exist $V_1, V_2 \dots \subseteq Z$ and $n_1 \le n_2 \le \cdots$ in
$\NN$ such that
\begin{enumerate}[(i)]
\item \label{it:V1} the $V_i$ are compact open bisections with mutually disjoint
    ranges and sources;
\item \label{it:V2} $\bigcup^j_{l=1} U_l \subseteq \big[\bigcup^{n_j}_{i=1}
    r(V_i)\big]_G$ and $\bigcup^j_{l=1} W_l \subseteq \big[\bigcup^{n_j}_{i=1}
    s(V_i)\big]_H$ for all $j \in \NN$; and
\item \label{it:V3} $r(V_i) \cap U_j = \emptyset$ and $s(V_i) \cap W_j = \emptyset$
    for all $j \in \NN$ and $i > n_j$.
\end{enumerate}

We construct the $V_i$ iteratively. Suppose either that $J=0$, or that $J\ge 1$, $n_1 \le
n_2 \le \cdots \le n_J\in \NN$, and $V_1,V_2,\dots,V_{n_J} \subseteq Z$ satisfy
(\ref{it:V1})--(\ref{it:V3}) for all $j < J$.

The set $K := U_{J+1} \setminus \big[\bigcup^{n_J}_{i=1} r(V_i)\big]_G$ is compact. Fix
$\Gamma \subseteq Z$ with $r(\Gamma) = K$. Suppose that $J\ge 1$ and $s(\Gamma) \cap
\bigcup^J_{l=1} W_l \not= \emptyset$, say $\gamma \in \Gamma \cap s^{-1} \big(
\bigcup^J_{l=1} W_l \big)$. Then $\bigcup^J_{l=1} W_l \subseteq \big[\bigcup^{n_J}_{i=1}
s(V_i)\big]_H$ gives $s(\gamma) \in \big[\bigcup^{n_J}_{i=1} s(V_i)\big]_H$, so there
exist $i \le n_J$, $\alpha \in H^{s(\gamma)}$ and $\beta \in V_i$ with $s(\alpha) =
s(\beta)$. But then $\beta\alpha^{-1}\gamma^{-1} \in GK \cap r(V_i)G$, contradicting the
definition of $K$. So $s(\Gamma) \cap \bigcup^J_{l=1} W_l = \emptyset$. Similarly, if
$s(\gamma) = s(\beta)$ for some $\gamma \in \Gamma$ and $\beta \in V_i$ where $i \le
n_J$, then $\gamma\beta^{-1} \in KG \cap G r(V_i)$, which is impossible by definition of
$K$; so $s(\Gamma) \cap s(V_i) = \emptyset$ for $i \le n_J$. We also have $r(\Gamma) = K
\subseteq \bigcup_{l > J} U_l \setminus \bigcup^{n_J}_{i=1} r(V_i)$, so for each $\gamma
\in \Gamma$, there is a compact open bisection $V^0_\gamma \subseteq Z$ containing
$\gamma$ with $r(V^0_\gamma) \cap U_l = \emptyset = s(V^0_\gamma) \cap W_l$ for $l \le
J$, and $r(V^0_\gamma) \cap r(V_i) = \emptyset = s(V^0_\gamma) \cap s(V_i)$ for $i \le
n_J$. Since $K$ is compact, there are $V^0_1, \dots, V^0_m \in \{V_\gamma : \gamma \in
\Gamma\}$ with $K \subseteq \bigcup^m_{i=1} r(V^0_i)$. For $i \le m$, let $V^1_i := V^0_i
\setminus r^{-1}\big(\bigcup_{i' < i} r(V^0_{i'})\big)$; so $K \subseteq
\bigsqcup^m_{i=1} r(V^1_i)$. Let $V_{n_J + 1} := V^1_1$ and iteratively put $V_{n_J + i}
:= V^1_i \setminus s^{-1}\big(\bigcup_{i' < i} s(V_{n_J+i'})\big)$. Then $V_1, \dots,
V_{n_J+m}$ are compact open bisections with mutually disjoint ranges and sources such
that $r(V_i) \cap U_j = \emptyset$ for $j\le J$ and $i > n_j$.

We claim that $K \subseteq \big[\bigcup^m_{i=1} r(V_{n_J +i})\big]_G$. Fix $x \in K$. Then
there are $i \le m$ and $\alpha \in V^1_i$ with $x = r(\alpha)$. By definition of
$V_{n_J+i}$ there exists $1 \le i' \le i$ and $\beta \in V_{n_J +i'}$ with $s(\beta) =
s(\alpha)$. So $\alpha\beta^{-1} \in G^x \cap G \in r(V_{n_j + i'})$, forcing $x \in
\big[\bigcup^{n_j}_{i=1} r(V_i)\big]_G$.

Now let $K' = W_{J+1} \setminus \big[\bigcup^{n_J+m}_{i=1} s(V_i)\big]_H$. We repeat the
argument of the previous two paragraphs. Choose $\Lambda \subseteq Z$ with $s(\Lambda) =
K'$. As above, $r(\Lambda) \cap \big(\bigcup^J_{l=1} U_l \cup \bigcup^{n_J+m}_{i=1}
r(V_i)\big) = \emptyset$. For $\lambda \in \Lambda$ pick a compact open bisection
$V^0_\lambda \subseteq Z$ containing $\lambda$ with $r(V^0_\lambda) \cap
\big(\bigcup^J_{l=1} U_l \cup \bigcup^{n_J+m}_{i=1} r(V_i)\big) = \emptyset$, and
$s(V^0_\lambda) \cap \big(\bigcup^J_{l=1} W_l \cup \bigcup^{n_J+m}_{i=1} s(V_i)\big) =
\emptyset$. Use compactness and disjointify sources to obtain $V^1_{m+1}, \dots
V^1_{m+m'}$ with $K' \subseteq \bigsqcup^{m'}_{i=1} s(V^1_{m+i})$. Iteratively let
$V_{n_J + m + i} := V^1_{m+i} \setminus r^{-1}\big(\bigcup_{i' < i} r(V_{n_J + m +
i'})\big)$. As for $K$ above, $K' \subseteq \big[\bigcup^{n_J+m+p}_{i=1} s(V_i)\big]_H$.
Let $n_{J+1} = n_J + m + m'$. Then $V_1, \dots V_{n_{J+1}}$ satisfy
(\ref{it:V1})--(\ref{it:V3}) for $j < J+1$. The claim now follows by induction.

Now let $Y := \bigcup^\infty_{i=1} V_i$. Then~(\ref{it:V1}) guarantees that $Y$ is an
open bisection. By~(\ref{it:V2}), $r(Y)$ is $G$-full and $s(Y)$ is $H$-full. Since each
$V_i$ is a compact open bisection, the $r(V_i)$ and $s(V_i)$ are clopen. So $r(Y)$ and
$s(Y)$ are open. They are also closed: by~(\ref{it:V3}), each $U_j \setminus r(Y) = U_j
\setminus \bigcup^{n_j}_{i=1} r(V_i)$ is open, and likewise and each $W_j \setminus s(Y) = W_j \setminus
\bigcup^{n_j}_{i=1} s(V_i)$ is open; so $\go \setminus r(Y) = \bigcup_j U_j \setminus
r(Y)$ and $\ho \setminus s(Y) = \bigcup_j W_j \setminus s(Y)$ are open.

The map $\gamma \mapsto Y^{-1} \gamma Y$ from $G|_{r(Y)}$ to $H|_{s(Y)}$ is a groupoid
isomorphism just as in the proof of Proposition~\ref{prp:corner->stable iso}. Hence $G$
is Kakutani equivalent to $H$, giving (\ref{it:ge})${\implies}$(\ref{it:Ke}).
\end{proof}

\begin{cor}
Let $G$ and $H$ be ample groupoids with $\sigma$-compact unit spaces. Suppose that there
exists a $G$-full compact open subset of $\go$. Then $G$ and $H$ are groupoid equivalent
if and only if there are full compact open sets $X \subseteq \go$ and $Y \subseteq \ho$
such that $G|_X \cong H|_Y$.
\end{cor}
\begin{proof}
Let $X_0 \subseteq \go$ be a $G$-full compact open set. First suppose $G$ and $H$ are
groupoid equivalent; say $Z$ is a $G$--$H$ equivalence. Following the construction of
$V_1$ in the proof of Theorem~\ref{thm:equiv equiv}---with $U_1 = X_0$---gives a compact
open bisection $V \subseteq Z$ with $X_0 \subseteq [r(V)]$, and hence $\go = [X_0]
\subseteq [r(V)]$. Fix $y \in \ho$. Take $\gamma \in Z_y$. Take $\alpha \in
G_{r(\gamma)}$ with $r(\alpha) \in r(V)$; say $\beta \in V \cap Z^{r(\alpha)}$. Then
$\beta^{-1}\alpha\gamma \in s(V)H \cap H_y$. So $[s(V)] = \ho$. Now $X := r(V) \subseteq
\go$ and $Y := s(V) \subseteq \ho$ are full compact open sets and $\gamma \mapsto
V^{-1}\gamma V : G|_{X} \to H_{Y}$ is an isomorphism. This proves the ``${\implies}$''
direction. The ``${\Longleftarrow}$'' direction follows from
(\ref{it:wKe})${\implies}$(\ref{it:ge}) in Theorem~\ref{thm:equiv equiv}.
\end{proof}

\section{Consequences for graph algebras}\label{sec:consequences}

In this section we explore the consequences of Theorem~\ref{thm:groupoidBGR} for stable
isomorphism of graph $C^*$-algebras and of Leavitt path algebras.  First we introduce some terminology to state our main result.

If $E$ is a directed graph, then $SE$ denotes the graph obtained by appending a head
$\dots f_{3,v} f_{2,v} f_{1,v}$ at every vertex $v$. Theorem~4.2 of \cite{Tomforde} shows
that $C^*(SE) \cong C^*(E) \otimes \Kk$ (see also
\cite[Proposition~9.8]{AbramsTomforde}). We will show in Lemma~\ref{lem:stabgraph} that
this happens at the level of groupoids. First, we briefly describe the graph groupoid
$G_E$: if $E^\infty$ denotes the set of all infinite paths of $E$ and $E^*$ denotes the
set of all finite paths of $E$, define
\[
\partial E := E^\infty \cup \{x \in E^* : \text{ $r(x)$ is a sink or an infinite emitter}\}.
\]
For $\mu \in E^*$, define $Z(\mu) := \{\mu x : x \in \partial E, r(\mu) = s(x)\}$. Then
the sets $Z(\mu \setminus F) := Z(\mu) \setminus \bigcup_{\nu \in F} Z(\mu\nu)$ indexed
by $\mu \in E^*$ and finite subsets $F$ of $r(\mu)E^1$ form a basis of compact open sets
for a locally compact Hausdorff topology on $\partial E$. For each $n \ge 0$, the
\emph{shift map} $\sigma^n : \partial E^{\ge n} := \{x \in \partial E : |x| \ge n\} \to
\partial E$ given by $\sigma^n(\mu x) = x$ for $\mu \in E^n$ and $x \in r(\mu)\partial E$
is a local homeomorphism.

We write $G_E$ for the graph groupoid
\[\textstyle
G_E = \bigcup_{m,n \in \NN} \{(x,m-n,y) : x \in \partial E^{\ge m}, y \in \partial E^{\ge n}\text{ and }\sigma^m(x) = \sigma^n(y)\},
\]
where $r(x,m,y) = (x,0,x)$, $s(x,m,y) = (y,0,y)$ and $(x,m,y)(y,n,z) = (x, m+n, z)$. This
is an ample groupoid under the topology with basic open sets $Z(\alpha,\beta \setminus F)
:= \{(\alpha x, |\alpha|-|\beta|, \beta x) : x \in Z(r(\alpha) \setminus F)\}$ indexed by
triples $(\alpha,\beta, F)$ where $\alpha,\beta \in E^*$, and $F \subseteq r(\alpha) E^1$ is finite.

\begin{lem}\label{lem:stabgraph}
Let $E$ be a directed graph. Then $G_E \times \Rr \cong G_{SE}$ and $G_{SE} \cong G_{SE}
\times \Rr$.
\end{lem}
\begin{proof}
For each $v \in E^0$, write $\mu_{0,v} := v$ and for $i \ge 1$ write $\mu_{i,v} :=
f_{i,v} f_{i-1,v} \dots f_{1,v}$. Then $\partial(SE) = \{\mu_{i,s(x)}x : x \in
\partial E, i \in \NN\}$. The map $\phi : \mu_{i,s(x)} x \mapsto (x, i)$ is a homeomorphism from
$\partial(SE)$ to $\partial E \times \NN$: cylinder sets of the form
$Z(\mu_{i,s(\lambda)}\lambda \setminus F)$ are a basis of compact open sets for
$\partial(SE)$, and $\phi$ restricts to a continuous bijection of each $Z(\mu_{i,
s(\lambda)}\lambda \setminus F)$ onto the compact open set $Z(\lambda \setminus F) \times
\{i\}$. It is routine to check that $\big((x,m,y),(i,j)\big) \mapsto \big(\phi^{-1}(x,i),
m+i-j, \phi^{-1}(y,j)\big)$ is a groupoid isomorphism from $G_E \times \Rr$ to $G_{SE}$.
Since $\Rr \times \Rr \cong \Rr$, we obtain $G_{SE} \cong G_{SE} \times \Rr$ as well.
\end{proof}

Recall \cite{BCW, MatsumotoMatui} that graphs $E$ and $F$ are \emph{orbit equivalent} if
there exist a homeomorphism $h : \partial E \to \partial F$ and continuous functions $k,l
: \partial E^{\ge 1} \to \NN$ and $k',l' : \partial F^{\ge 1} \to \NN$ such
that $\sigma^{k(x)}_F(h(\sigma_E(x))) = \sigma^{l(x)}_F(h(x))$ and
$\sigma^{k'(y)}_E(h^{-1}(\sigma_F(y))) = \sigma^{l'(y)}_E(h^{-1}(y))$ for all $x \in
\partial E^{\ge 1}$ and $y \in \partial F^{\ge 1}$.

We assume familiarity with graph $C^*$-algebras and Leavitt path algebras; see
\cite{BHRS} and \cite{Tomforde2} for the requisite background. Given a graph $E$, we call
the abelian subalgebra $D(E) := \clsp\{s_\mu s^*_\mu : \mu \in E^*\} \subseteq C^*(E)$
the \emph{diagonal subalgebra} of the graph $C^*$-algebra, and for any commutative ring
$R$ with~1, we call the abelian subalgebra $D_R(E) := \lsp_R\{s_\mu s_{\mu^*} : \mu \in
E^*\} \subseteq L_R(E)$ the diagonal subalgebra of the Leavitt path $R$-algebra. For an
ample groupoid $G$, we write $C^*(G)$ for the (full) $C^*$-algebra of $G$ (see for
example \cite{Renault80} or \cite{Paterson}) and, for a commutative ring $R$ with~1, we
write $A_R(G)$ for the Steinberg algebra of $G$ over $R$ (see \cite{Steinberg} and
\cite{ClarkSims:JPAA15}). Any isomorphism of ample groupoids $G$ and $H$ induces an isomorphism $C^*(G)\cong C^*(H)$ carrying $C_0(G^{(0)})$ to $C_0(H^{(0)})$ and an isomorphism $A_R(G)\cong A_R(H)$ carrying $A_R(G^{(0)})$ to $A_R(H^{(0)})$.
The canonical isomorphism $C^*(E) \cong C^*(G_E)$ carries
$D(E)$ to the standard diagonal subalgebra $C_0(\go_E) \subseteq C^*(G_E)$ (see the proof
of \cite[Proposition~4.1]{KPRR} and \cite[Proposition 2.2]{BCW}), and likewise at the
level of Leavitt path algebras \cite[Example~3.2]{ClarkSims:JPAA15}. We say that an
isomorphism $\phi : C^*(E) \to C^*(F)$ of graph $C^*$-algebras is \emph{diagonal
preserving} if $\phi(D(E)) = D(F)$, and likewise for Leavitt path algebras.

We write $\Kk$ for the $C^*$-algebra of compact operators on $\ell^2(\NN)$, and $\Cc$ for
the maximal abelian subalgebra of $\Kk$ consisting of diagonal operators. For a
commutative ring $R$ with~1 we write $M_\infty (R)$ for the ring of finitely supported,
countably infinite square matrices over $R$ and $D_\infty (R)$ for the abelian subring of
$M_\infty(R)$ consisting of diagonal matrices. For any ample groupoid $G$, if $\Rr$ is
the equivalence relation $\NN \times \NN$ of Section~\ref{sec:BGR}, there exist
isomorphisms $C^*(G \times \Rr) \cong C^*(G) \otimes \Kk$ and $A_R(G \times \Rr) \cong
A_R(G) \otimes M_\infty(R)$ that take $C_0(\go \times \NN)$ to $C_0(\go) \otimes \Cc$ and
$A_R(\go \times \NN)$ to $A_R(\go) \otimes D_\infty(R)$.

An isomorphism $\phi : C^*(E)\otimes\Kk \to C^*(F)\otimes\Kk$ is \emph{diagonal
preserving} if $\phi(D(E)\otimes\Cc) = D(F)\otimes\Cc$, and similarly at the level of
Leavitt path algebras.

\begin{thm}\label{thm:main}
Let $E$ and $F$ be directed graphs, and let $R$ be a commutative integral domain with~1.
The following are equivalent:
\begin{enumerate}
\item\label{it:CE sdcong CF} there is a diagonal-preserving isomorphism $C^*(E)
    \otimes \Kk \cong C^*(F) \otimes \Kk$;
\item\label{it:LE sdcong LF} there is a diagonal-preserving $^*$-ring isomorphism
    $L_R(E) \otimes M_\infty(R) \cong L_R(F) \otimes M_\infty(R)$;
\item\label{it:CSE dcong CSF} there is a diagonal-preserving isomorphism $C^*(SE)
    \cong C^*(SF)$;
\item\label{it:LSE dcong LSF} there is a diagonal-preserving $^*$-ring isomorphism
    $L_R(SE) \cong L_R(SF)$;
\item\label{it:E sge F} $G_E \times \Rr \cong G_F \times \Rr$;
\item\label{it:SE ge SF} $G_{SE} \cong G_{SF}$. \setcounter{saver}{\value{enumi}}
\end{enumerate}
These equivalent conditions imply each of
\begin{enumerate}\setcounter{enumi}{\value{saver}}
\item\label{it:SE oe SF} $SE$ and $SF$ are orbit equivalent;
\item\label{it:LE sdcong LF'} there is a diagonal-preserving ring isomorphism $L_R(E)
    \otimes M_\infty(R) \cong L_R(F) \otimes M_\infty(R)$; and
\item\label{it:LSE dcong LSF'} there is a diagonal-preserving ring isomorphism
    $L_R(SE) \cong L_R(SF)$.
\end{enumerate}
The conditions (\ref{it:LE sdcong LF'})~and~(\ref{it:LSE dcong LSF'}) are equivalent. If
every cycle in each of $E$ and $F$ has an exit, then (\ref{it:CE sdcong
CF})--(\ref{it:LSE dcong LSF'}) are all equivalent.
\end{thm}

Our proof of Theorem~\ref{thm:main} uses Crisp and Gow's collapsing procedure
\cite{CrispGow}.

\begin{lem}\label{lem:collapse}
Let $E$ be a directed graph, let $T$ be a collapsible subgraph of $E$ in the sense of
Crisp and Gow, and let $F$ be the graph obtained from $E$ by collapsing $T$. Then $G_E
\times \Rr \cong G_F \times \Rr$, and there are diagonal-preserving isomorphisms $C^*(E)
\otimes \Kk \cong C^*(F) \otimes \Kk$ and $L_R(E) \otimes M_\infty(R) \cong L_R(F)
\otimes M_\infty(R)$ for every commutative unital ring $R$.
\end{lem}
\begin{proof}
Proposition~6.2 of \cite{ClarkSims:JPAA15} shows that $G_E$ and $G_F$ are equivalent
groupoids. So Theorem~\ref{thm:groupoidBGR} gives $G_E \times \Rr \cong G_F \times \Rr$.
The result then follows because the canonical isomorphisms $C^*(E) \otimes \Kk \cong C^*(G_E) \otimes \Kk$ and
$L_R(E) \otimes M_\infty(R) \cong A_R(G_E) \otimes M_\infty(R)$ are diagonal preserving.
\end{proof}

\begin{proof}[Proof of Theorem~\ref{thm:main}]
Lemma~\ref{lem:stabgraph} yields (\ref{it:E sge F})${\iff}$(\ref{it:SE ge SF}); and also
(\ref{it:CE sdcong CF})${\iff}$(\ref{it:CSE dcong CSF}) and (\ref{it:LE sdcong
LF})${\iff}$(\ref{it:LSE dcong LSF}) since the canonical isomorphisms $C^*(E) \cong
C^*(G_E)$ and $L_R(E) \cong A_R(G_E)$ are diagonal preserving. Theorem~5.1 of \cite{BCW}
gives (\ref{it:CSE dcong CSF})${\iff}$(\ref{it:SE ge SF}). Isomorphisms of groupoids
induce diagonal-preserving $^*$-ring isomorphisms of Steinberg algebras, giving
(\ref{it:SE ge SF})${\implies}$(\ref{it:LSE dcong LSF}). To prove equivalence of
(\ref{it:CE sdcong CF})--(\ref{it:SE ge SF}), it now suffices to check (\ref{it:LE sdcong
LF})${\implies}$(\ref{it:E sge F}).

Suppose that~(\ref{it:LE sdcong LF}) holds. The graphs $E$ and $F$ can be obtained from
their Drinen--Tomforde desingularisations $E'$ and $F'$ by applications of Crisp and
Gow's collapsing procedure \cite{CrispGow}. So Lemma~\ref{lem:collapse} gives $G_E \times
\Rr \cong G_{E'} \times \Rr \cong G_{SE'}$, and similarly for $F$. Hence the
diagonal-preserving $^*$-ring isomorphism $L_R(E) \otimes M_\infty(R) \cong L_R(F)
\otimes M_\infty(R)$ induces a diagonal-preserving $^*$-ring isomorphism $L_R(SE') \cong
L_R(SF')$. Since $SE'$ and $SF'$ are row-finite with no sinks, we can apply
\cite[Theorem~6.2]{BCaH}\footnote{ Theorem~6.2 of \cite{BCaH} says ``no sources" rather
than ``no sinks" as they use the convention that $s_e^*s_e = p_{s(e)}$ rather than
$s_e^*s_e = p_{r(e)}$.} to obtain $G_{SE'} \cong G_{SF'}$. Since $G_E \times \Rr \cong
G_{SE'}$ and $G_F \times \Rr \cong G_{SF'}$, this yields $G_E \times \Rr \cong G_F \times
\Rr$.

That (\ref{it:SE ge SF})${\implies}$(\ref{it:SE oe SF}) follows from \cite[Theorem~5.1
(2)${\implies}$(4)]{BCW}. Clearly (\ref{it:LE sdcong LF})${\implies}$(\ref{it:LE sdcong
LF'}) and (\ref{it:LSE dcong LSF}){$\implies$}(\ref{it:LSE dcong LSF'}). We have
(\ref{it:LE sdcong LF'})${\iff}$(\ref{it:LSE dcong LSF'}) by another application of
Lemma~\ref{lem:stabgraph}.

Now suppose that every cycle in each of $E$ and $F$ has an exit. Then \cite[Theorem~5.1
(4)${\implies}$(2)]{BCW} gives (\ref{it:SE oe SF})${\implies}$(\ref{it:SE ge SF}), and
\cite[Corollary~4.4]{ABHS} gives (\ref{it:LSE dcong LSF'})${\implies}$(\ref{it:SE ge
SF}).
\end{proof}

We now deduce a ``diagonal-preserving'' version of \cite[Corollary~2.6]{Brown:PJM77} for
groupoid $C^*$-algebras and Steinberg algebras from Proposition~\ref{prp:corner->stable
iso}. We do not require that $G$ is Hausdorff (see for example \cite{Paterson} for the
definition of the $C^*$-algebra of a non-Hausdorff groupoid, and \cite{Steinberg} for the
definition of the Steinberg algebra of a non-Hausdorff groupoid).

\begin{lem}\label{lem:corner->stable iso}
Let $G$ be an ample groupoid such that $\go$ is $\sigma$-compact and let $R$ be a ring.
Suppose that $K\subseteq \go$ is clopen and $G$-full. Then
\begin{enumerate}
\item\label{it:C*stable} there is an isomorphism $\phi:C^*(G)\otimes\Kk\to
    C^*(G|_K)\otimes\Kk$ such that $\phi(C_0(G^{(0)})\otimes\Cc)=C_0(K)\otimes\Cc$;
    and
\item\label{it:SteinbergStable} there is a $^*$-ring isomorphism $\eta:A_R(G)\otimes
    M_\infty (R)\to A_R(G|_K)\otimes M_\infty (R)$ such that
    $\eta(A_R(G^{(0)})\otimes D_\infty (R))=A_R(K)\otimes D_\infty (R)$.
\end{enumerate}
\end{lem}
\begin{proof}
The canonical isomorphisms $C^*(G \times \Rr)\cong C^*(G)\otimes\Kk$ and $C^*(G|_K \times
\Rr)\cong C^*(G|_K)\otimes\Kk$ carry $C_0(G^{(0)}\times\NN)$ to $C_0(G^{(0)})\otimes\Cc$
and $C_0(K\times\NN)$ to $C_0(K)\otimes\Cc$. Similarly, the canonical $^*$-ring
isomorphisms $A_R(G \times \Rr)\cong A_R(G)\otimes M_\infty (R)$ and $A_R(G|_K \times
\Rr)\cong A_R(G|_K)\otimes M_\infty (R)$ carry $A_R(G^{(0)}\times\NN)$ to
$A_R(G^{(0)})\otimes D_\infty (R)$ and $A_R(K\times\NN)$ to $A_R(K)\otimes D_\infty (R)$.
Hence both statements follow from Proposition~\ref{prp:corner->stable iso}.
\end{proof}

From Lemma~\ref{lem:corner->stable iso}, Theorem~\ref{thm:equiv equiv} and
Theorem~\ref{thm:main} we obtain a version of \cite[Theorem~5.4]{Matui} for graph
$C^*$-algebras and Leavitt path algebras. For a ring $A$, we denote by $M(A)$ the
\emph{multiplier ring} of $A$ (see for example \cite{AP}).

\begin{cor}\label{cor:matui}
Let $E$ and $F$ be directed graphs, and let $R$ be a commutative integral domain with~1.
The following are equivalent:
\begin{enumerate}
\item\label{it:matuiKe} $G_E$ and $G_F$ are Kakutani equivalent;
\item\label{it:C*corners} there exist projections $p_E\in M(D(E))$ and $p_F\in
    M(D(F))$ and an isomorphism $\phi:p_EC^*(E)p_E\to p_FC^*(F)p_F$ such that $p_E$
    is full in $C^*(E)$, $p_F$ is full in $C^*(F)$, and $\phi(p_ED(E))=p_FD(F)$;
\item\label{it:LPAcorners} there exist projections $p_E\in M(D_R(E))$ and $p_F\in
    M(D_R(F))$ and a $^*$-ring isomorphism $\eta:p_EL_R(E)p_E\to p_FL_R(F)p_F$ such
    that $p_E$ is full in $L_R(E)$, $p_F$ is full in $L_R(F)$, and
    $\eta(p_ED_R(E))=p_FD_R(F)$.
\end{enumerate}
\end{cor}
\begin{proof}
We prove (\ref{it:matuiKe})${\iff}$(\ref{it:C*corners}); the proof of
(\ref{it:matuiKe})${\iff}$(\ref{it:LPAcorners}) is similar.

First suppose~(\ref{it:matuiKe}); say $X\subseteq G_E^{(0)}$ is a $G_E$-full clopen
subset and $Y\subseteq G_F^{(0)}$ is a $G_F$-full clopen subset such that $(G_E)|_X\cong
(G_F)|_Y$. Then the characteristic function of $X$ corresponds to a projection $p_E\in
M(D(E))$ which is full in $C^*(E)$ and such that $C^*((G_E)|_X)\cong p_EC^*(E)p_E$ by an
isomorphism that maps $C_0(X)$ onto $p_ED(E)$. Similarly, the characteristic function of
$Y$ corresponds to a projection $p_F\in M(D(F))$ which is full in $C^*(F)$ and such that
$C^*((G_F)|_Y)\cong p_FC^*(F)p_F$ by an isomorphism that maps $C_0(Y)$ onto $p_FD(F)$.
The isomorphism $(G_E)|_X\cong (G_F)|_Y$ gives an isomorphism $C^*((G_E)|_X)\cong
C^*((G_F)|_Y)$ that maps $C_0(X)$ onto $C_0(Y)$, which yields~(\ref{it:C*corners}).

Now suppose (\ref{it:C*corners}). The projection $p_E\in M(D(E))$ corresponds to a
$G_E$-full clopen subset $X$ of $G_E^{(0)}$ such that there is an isomorphism
$C^*((G_E)|_X)\cong p_EC^*(E)p_E$ that maps $C_0(X)$ onto $p_ED(E)$, and the projection
$p_F\in M(D(F))$ corresponds to a $G_F$-full clopen subset $Y$ of $G_F^{(0)}$ such that
there is an isomorphism $C^*((G_F)|_Y)\cong p_FC^*(F)p_F$ that maps $C_0(Y)$ onto
$p_FD(F)$. Lemma~\ref{lem:corner->stable iso}(\ref{it:C*stable}) gives a
diagonal-preserving isomorphism $C^*(E) \otimes \Kk \cong C^*(F) \otimes \Kk$. So
Theorem~\ref{thm:main} implies that $G_E \times \Rr$ and $G_F \times \Rr$ are groupoid
equivalent, and hence Theorem~\ref{thm:equiv equiv} implies that they are Kakutani
equivalent.
\end{proof}

Theorem~\ref{thm:equiv equiv} implies that the equivalent conditions
(\ref{it:matuiKe})--(\ref{it:LPAcorners}) of Corollary~\ref{cor:matui} are also
equivalent to the equivalent conditions (\ref{it:CE sdcong CF})--(\ref{it:SE ge SF}) of
Theorem~\ref{thm:main}.

\begin{cor}\label{cor:Z-algs}
If $E$ and $F$ are directed graphs, then $L_\ZZ(E) \otimes M_\infty(\ZZ) \cong L_\ZZ(F)
\otimes M_\infty(\ZZ)$ as $^*$-rings if and only if there is a diagonal-preserving
isomorphism $C^*(E) \otimes \Kk \cong C^*(F) \otimes \Kk$.
\end{cor}
\begin{proof}
First suppose that $L_\ZZ(E) \otimes M_\infty(\ZZ) \cong L_\ZZ(F) \otimes M_\infty(\ZZ)$
as $^*$-rings. Then $L_\ZZ(SE) \cong L_\ZZ(SF)$ as $^*$-rings as well. By
\cite[Corollary~6]{Carlsen}, this $^*$-isomorphism is diagonal preserving, so
(2)${\implies}$(1) of Theorem~\ref{thm:main} gives a diagonal-preserving isomorphism
$C^*(E) \otimes \Kk \cong C^*(F) \otimes \Kk$. The reverse implication follows from
(1)${\implies}$(2) of Theorem~\ref{thm:main}.
\end{proof}

An important question about Leavitt path algebras is whether the complex Leavitt path
algebras of the graphs
\[
\begin{tikzpicture}
    \node[anchor=east] at (-0.3,0.25) {$E_2 = {}$};
    \node[circle, fill=black, inner sep=1pt] (v) at (0,0) {};
    \draw[-stealth] (v) .. controls +(1,1) and +(-1,1) .. (v);
    \draw[-stealth] (v) .. controls +(0.75, 0.75) and +(-0.75,0.75) .. (v);
    \node at (2.4,0.3) {and};
    \node[anchor=east] at (5.7,0.25) {$E_{2-} = {}$};
    \node[circle, fill=black, inner sep=1pt] (w1) at (6,0) {};
    \node[circle, fill=black, inner sep=1pt] (w2) at (7.5,0) {};
    \node[circle, fill=black, inner sep=1pt] (w3) at (9,0) {};
    \draw[-stealth] (w1) .. controls +(1,1) and +(-1,1) .. (w1);
    \draw[-stealth] (w1) .. controls +(0.75,0.75) and +(-0.75,0.75) .. (w1);
    \draw[-stealth] (w2) .. controls +(1,1) and +(-1,1) .. (w2);
    \draw[-stealth] (w3) .. controls +(1,1) and +(-1,1) .. (w3);
    \draw[-stealth, out=30, in=150] (w1) to (w2);
    \draw[-stealth, out=30, in=150] (w2) to (w3);
    \draw[-stealth, out=210, in=330] (w3) to (w2);
    \draw[-stealth, out=210, in=330] (w2) to (w1);
\end{tikzpicture}
\]
are isomorphic. This was recently answered in the negative for Leavitt path algebras over
$\ZZ$ as $^*$-rings \cite{JohansenSorensen}. We extend this to the question of stable
$^*$-isomorphism.

\begin{cor}\label{cor:graph over Z}
Let $E$ and $F$ be strongly connected finite graphs such that $L_\ZZ(E) \otimes
M_\infty(\ZZ)$ and $L_\ZZ(F) \otimes M_\infty(\ZZ)$ are $^*$-isomorphic. Then $\det(1 -
A_E^t) = \det(1 - A_F^t)$. In particular, $L_\ZZ(E_2) \otimes M_\infty(\ZZ)$ and
$L_\ZZ(E_{2-}) \otimes M_\infty(\ZZ)$ are not $^*$-isomorphic.
\end{cor}
\begin{proof}
Corollary~\ref{cor:Z-algs} gives a diagonal-preserving isomorphism $C^*(E) \otimes \Kk
\cong C^*(F) \otimes \Kk$. If $E$ and $F$ have cycles with exits, then as discussed in
the proof of \cite[Corollary~3.8]{MatsumotoMatui}, the proof of
\cite[Theorem~4.1]{Matsumoto} combined with \cite[Theorem~3.6]{MatsumotoMatui} gives
$\det(1 - A_E^t) = \det(1 - A_F^t)$.  If $E$ and $F$ have cycles without exists, then
$A_E$ and $A_F$ are permutation matrices, so $\det(1 - A_E^t) = \det(1 - A_F^t) = 0$. To
prove the final statement, one checks that $\det(1 - A_{E_2}^t) = -1$ and $\det(1 -
A_{E_{2-}}^t) = 1$.
\end{proof}

\begin{rmk}
It is natural to ask whether Corollary~\ref{cor:graph over Z} can be used to decide
whether $L_\ZZ(E_2)$ and $L_\ZZ(E_{2-})$ are Morita equivalent. Theorem~5 and part~2 of
the remarks following Corollary~7 in \cite{AAM} show that rings with enough idempotents
are stably isomorphic if and only if they are Morita equivalent. But this is a result
about ring isomorphisms, whereas Corollaries \ref{cor:Z-algs}~and~\ref{cor:graph over Z}
are about $^*$-ring isomorphisms (and the $^*$-preserving hypothesis is crucial to the
argument of \cite[Corollary~6]{Carlsen}, upon which our results hinge). So the question
remains open whether $L_\ZZ(E_2)$ and $L_\ZZ(E_{2-})$ are Morita equivalent. There is a
notion of Morita $^*$-equivalence for rings \cite{Ara:ART1999}. Though we were unable to
locate a reference, it seems likely that an analogue of \cite[Theorem~5]{AAM} holds for
stable $^*$-isomorphism and Morita $^*$-equivalence. If so, then such a result could be
combined with Corollary~\ref{cor:graph over Z} to prove that $L_\ZZ(E_2)$ and
$L_\ZZ(E_{2-})$ are not Morita $^*$-equivalent.
\end{rmk}

Theorem~\ref{thm:main} has implications for the stable isomorphisms associated to
S{\o}rensen's move equivalences of graphs \cite{Sorensen}. \emph{Move equivalence} for
graphs with finitely many vertices is the equivalence relation generated by four
operations: deleting a regular source; collapsing a regular vertex; in-splitting at a
regular vertex; and outsplitting. By \cite[Theorem~4.3]{Sorensen}, if $C^*(E)$ and
$C^*(F)$ are simple and $E$ and $F$ each contain at least one infinite emitter, then
$C^*(E) \otimes \Kk \cong C^*(F) \otimes \Kk$ if and only if $E$ and $F$ are move
equivalent.

\begin{cor}
Let $E$ and $F$ be directed graphs with finitely many vertices. Suppose that $E$ and $F$
are move equivalent. Then $G_E \times \Rr \cong G_F \times \Rr$, and there are
diagonal-preserving isomorphisms $C^*(E) \otimes \Kk \cong C^*(F) \otimes \Kk$ and
$L_R(E) \otimes M_\infty(R) \cong L_R(F) \otimes M_\infty(R)$ for every commutative ring
$R$ with 1.
\end{cor}
\begin{proof}
S{\o}rensen's moves (S), (R) and (I) are all examples of Crisp and Gow's collapsing
procedure (see page 2070--2071 of \cite{ClarkSims:JPAA15}), so if $F$ is obtained from $E$ by
applying any of these moves, then Lemma~\ref{lem:collapse} shows that $G_E \times \Rr
\cong G_F \times \Rr$. By \cite[Theorem~6.1 and Corollary~6.2]{BCW}, if $F$ is obtained
from $E$ by applying move~(O), then $G_E \cong G_F$, so certainly $G_E \times \Rr \cong
G_F \times \Rr$. Induction establishes that if $E$ and $F$ are move equivalent then $G_E
\times \Rr \cong G_F \times \Rr$. The remaining statements follow from
Theorem~\ref{thm:main}.
\end{proof}

\end{document}